\documentclass[a4paper,12pt]{article}
\usepackage{hyperref}

\usepackage{amssymb,url,amsthm}
\usepackage{mathdots}
\usepackage{mathtools}
\usepackage{xcolor,graphicx}
\usepackage{subcaption}

\theoremstyle{plain} 
\newtheorem{theorem}{Theorem}[section]
\newtheorem{lemma}[theorem]{Lemma}

\newtheorem{corollary}[theorem]{Corollary}

\theoremstyle{definition} 
\newtheorem{definition}[theorem]{Definition}
\newtheorem{remark}[theorem]{Remark}

\newcommand{\bsa}{\boldsymbol{a}}

\newcommand{\bsc}{\boldsymbol{c}}
\newcommand{\bse}{\boldsymbol{e}}
\newcommand{\bsg}{\boldsymbol{g}}

\newcommand{\bsk}{\boldsymbol{k}}

\newcommand{\bsx}{\boldsymbol{x}}
\newcommand{\bsy}{\boldsymbol{y}}

\newcommand{\bszero}{\boldsymbol{0}}
\newcommand{\bsone}{\boldsymbol{1}}
\newcommand{\bsdelta}{\boldsymbol{\delta}}
\newcommand{\FF}{\mathbb{F}}
\newcommand{\NN}{\mathbb{N}}
\newcommand{\RR}{\mathbb{R}}

\newcommand{\ZZ}{\mathbb{Z}}

\newcommand{\Scal}{\mathcal{S}}

\allowdisplaybreaks

\providecommand{\keywords}[1]
{
  \noindent \small	
  \textbf{Keywords:} #1
}
\providecommand{\amscode}[1]
{
  \noindent \small	
  \textbf{AMS subject classifications:} #1
}

\begin{document}

\title{On the quasi-uniformity properties of quasi-Monte Carlo point sets and sequences -- Part~II: digital nets and sequences
\thanks{The work of J.~D. is supported by ARC grant DP220101811.
The work of T.~G. is supported by JSPS KAKENHI Grant Number 23K03210.
The work of K.~S. is supported by JSPS KAKENHI Grant Number 24K06857.}}
\author{Josef Dick\thanks{School of Mathematics and Statistics, UNSW Sydney, Kensington, NSW 2052, Austrlia (\url{josef.dick@unsw.edu.au})}, 
Takashi Goda\thanks{Graduate School of Engineering, The University of Tokyo, 7-3-1 Hongo, Bunkyo-ku, Tokyo 113-8656, Japan (\url{goda@frcer.t.u-tokyo.ac.jp})},  
Kosuke Suzuki\thanks{Corresponding author, Faculty of Science, Yamagata University, 1-4-12 Kojirakawa-machi, Yamagata, 990-8560, Japan (\url{kosuke-suzuki@sci.kj.yamagata-u.ac.jp})}
}
\date{\today}
\maketitle

\keywords{Quasi-Monte Carlo, digital net, space-filling design, covering radius, separation radius, mesh ratio}
\amscode{05B40, 11K36, 11K45, 52C15, 52C17}

\begin{abstract}
We study the quasi-uniformity properties of digital nets, a class of quasi-Monte Carlo point sets. Quasi-uniformity is a space-filling property used for instance in experimental designs and radial basis function approximation. However, it has not been investigated so far whether common low-discrepancy digital nets are quasi-uniform, with the exception of the two-dimensional Sobol' sequence, which has recently been shown not to be quasi-uniform. 
In this paper, with the goal of constructing quasi-uniform low-discrepancy digital nets, we introduce the notion of \emph{well-separated} point sets and provide an algebraic criterion to determine whether a given sequence of digital nets is well-separated. Using this criterion, we present an example of a two-dimensional digital net which has low-discrepancy and is quasi-uniform. Additionally, we provide several counterexamples of low-discrepancy digital nets that are not quasi-uniform.
\end{abstract}

\sloppy

\section{Introduction}

It is natural to ask whether strong distributional uniformity in structured quasi-Monte Carlo (QMC) sampling \cite{DKS13,LePi14,DKP22,SJ94,DP10,N92} can be reconciled with geometric space-filling properties.
This work constitutes Part~II of our study on the two important structured QMC families: Part~I \cite{DGLPSxx} treats lattice-type point sets and Kronecker sequences, while the present paper focuses on digital nets and sequences in $[0,1]^d$.
We investigate when such digital constructions can simultaneously exhibit low-discrepancy features and \emph{quasi-uniformity}.
Here and throughout, we use quasi-uniformity to mean that the mesh ratio (the covering radius divided by the separation radius) is uniformly bounded, so the points neither cluster excessively nor leave large holes at the relevant scale; precise definitions are given in Section~2.2, and broader motivation and context are developed in Part~I \cite{DGLPSxx}.
Further motivation and applications can be found, for example, in the design of computer experiments \cite{FLS06,PM12,SWN03} and scattered data approximation \cite{Sch95,W05,SW06,T20,WBG21}.

Our main contributions are twofold.
First, we obtain predominantly \emph{negative} results: for several explicit and widely used families of digital $(0,m,d)$-nets (hence strong low-discrepancy guarantees),
we prove directly that they fail to be quasi-uniform.
In particular, their digital (algebraic) structure does not, by itself, enforce favorable geometric separation behavior.
This stands in sharp contrast to the lattice setting, where Part~I provides explicit positive examples in which low-discrepancy and quasi-uniformity coexist.

Second, we show that coexistence is nevertheless possible within the digital framework in dimension two.
To this end, we introduce the notion of \emph{well-separated} point sets tailored to digital nets, and we establish an algebraic criterion to verify well-separatedness for sequences of digital nets.
Using this criterion, we give an explicit two-dimensional example that is simultaneously low-discrepancy and quasi-uniform.
Constructing low-discrepancy, quasi-uniform digital nets and sequences in dimensions $d\ge 3$ remains open.

The rest of this paper is organized as follows. 
In the next section, we introduce the definition of digital nets and sequences, and explain the concept of quasi-uniformity for point sets in the unit cube, briefly reviewing some known results in this research direction. 
In Section~\ref{sec:well-separated}, to capture the quasi-uniformity of digital nets, we introduce the notion of \emph{well-separated} point sets and provide an algebraic criterion to determine whether a given sequence of digital nets is well-separated.
This section also includes an example of a two-dimensional digital net that has low-discrepancy and is quasi-uniform.
Section~\ref{sec:not-well-separated} provides several counterexamples of low-discrepancy digital nets that are not quasi-uniform, going beyond the example recently shown in \cite{G24a}, namely, that the two-dimensional Sobol' sequence is not quasi-uniform.

\section{Preliminaries}

In the following we will use $P, P_i, \ldots$ to denote generic point sets and $Q, Q_m, \ldots$ to denote QMC point sets or sequences.

\subsection{Digital nets and sequences}

Throughout this paper, let $b$ be a prime number and let $\FF_b$ be the finite field of order $b$. 
Digital nets are defined as follows:

\begin{definition}[Digital net]\label{def:digitalnet}
For a positive integer $m$, let $F_1, F_2, \ldots, F_d \in \FF_b^{m \times m}$ be matrices over the finite field $\FF_b$. 
For $n \in \{0, 1, 2, \ldots, b^m-1\}$ with $b$-adic expansion 
\[ n = n_0 + n_1 b + \cdots + n_{m-1} b^{m-1},\]
we set $\vec{n} = (n_0, n_1, \ldots, n_{m-1})^\top$, where we identify $n_0, n_1, \ldots, n_{m-1} \in \{0, 1, \ldots, b-1\}$ with the elements in $\FF_b$. 
Then, let $\vec{x}_{n,j} = (x_{n,j,1}, x_{n,j,2}, \ldots, x_{n,j,m})^\top = F_j \vec{n}$ and set 
\begin{equation} \label{eq:x_expand}
x_{n,j} = \frac{x_{n,j,1}}{b} + \frac{x_{n,j,2}}{b^{2}} + \cdots + \frac{x_{n,j,m}}{b^{m}}\in [0,1),
\end{equation}
and $\bsx_n = (x_{n,1}, x_{n,2}, \ldots, x_{n,d})^\top$. 
Then, we call $Q(F_1, \ldots, F_d) = \{\bsx_0, \bsx_1, \ldots, \bsx_{b^m-1}\}$ the digital net generated by $F_1, \ldots, F_d$ over $\FF_b$.
\end{definition}

We also consider digitally shifted digital nets, which are defined as follows:

\begin{definition}[Digitally shifted digital net]
For a positive integer $m$, let $Q(F_1, \ldots, F_d)$ be the digital net generated by $F_1, \ldots, F_d\in \FF_b^{m\times m}$. Let $\bsdelta = (\delta_1,\dots,\delta_d) \in \{0, 1/b^m, 2/b^m, \ldots, (b^m-1)/b^m\}^d$ with $b$-adic expansion
\[
\delta_j = \frac{\delta_{j,1}}{b} + \frac{\delta_{j,2}}{b^{2}} + \cdots + \frac{\delta_{j,m}}{b^{m}}
\] 
for $1 \le j \le d$.
Then, the digitally shifted digital net $Q(F_1, \ldots, F_d) \oplus \bsdelta$ (with depth $m$) is given by 
\begin{equation*}
Q(F_1, \ldots, F_d) \oplus \bsdelta = \{\bsx_0 \oplus \bsdelta, \bsx_1 \oplus \bsdelta, \ldots, \bsx_{b^m-1} \oplus \bsdelta \},
\end{equation*}
where $\bsx_n \oplus \bsdelta = (x_{n,1} \oplus \delta_1, \ldots, x_{n,d} \oplus \delta_d)^\top$ and
\begin{equation*}
x_{n,j} \oplus \delta_j := \frac{y_{n,j,1}}{b} + \cdots + \frac{y_{n,j,m}}{b^m},
\end{equation*}
and $y_{n,j,k} = x_{n,j,k} \oplus \delta_{j,k} = x_{n,j,k} + \delta_{j,k} \pmod{b}$, where $x_{n,j,k}$ is given as in \eqref{eq:x_expand}.
\end{definition}

To generate an infinite sequence of points, the definition of digital nets is extended as follows:
\begin{definition}[Digital sequence]\label{def:digitalseq}
Let $F_1, F_2, \ldots, F_d \in \FF_b^{\infty \times \infty}$ be matrices over the finite field, where $F_j=(f_{k,\ell}^{(j)})_{k,\ell \ge 1}$. 
We assume that, for all $j$ and $\ell$, we have $f_{k,\ell}^{(j)}=0$ for all sufficiently large $k$.
For an integer $n \ge 0$ with the $b$-adic expansion
\[ n = n_0 + n_1 b + \cdots ,\]
where the coefficients $n_0, n_1, \ldots\in \{0, 1, \ldots, b-1\}$ are identified with the elements in $\FF_b$, we set $\vec{n} = (n_0, n_1, \ldots )^\top\in \FF_b^{\infty}$.
Then, let $\vec{x}_{n,j} = (x_{n,j,1}, x_{n,j,2}, \ldots,)^\top = F_j \vec{n}$ and set 
\[ x_{n,j} = \frac{x_{n,j,1}}{b} + \frac{x_{n,j,2}}{b^{2}} + \cdots \in [0,1),\] 
and $\bsx_n = (x_{n,1}, x_{n,2}, \ldots, x_{n,d})^\top$. 
Then, we call $Q(F_1, \ldots, F_d) = \{\bsx_0, \bsx_1, \ldots \}$ the digital sequence generated by $F_1, \ldots, F_d$ over $\FF_b$.
\end{definition}

From the definitions of digital nets and sequences, it is clear that the generating matrices $F_1,\ldots,F_d$ determine how well-distributed the corresponding point set or sequence is.
The $t$-value, introduced in the next section, is one of the central quantities used to measure the quality of $F_1,\ldots,F_d$.
Digital nets and sequences with smaller $t$-values have a smaller upper bound on the star discrepancy.
However, merely looking at the $t$-value is not sufficient to determine whether a given digital net or sequence is quasi-uniform.

\subsection{Quasi-Uniformity}

In this part we follow the terminology and notation introduced in Part~I \cite[Section~1.1]{DGLPSxx}.
For convenience, we recall their notation below.
Throughout, we work on the unit cube $\Omega=[0,1]^d$; hence we omit the dependence on $\Omega$ from the notation.

Let $P$ be a point set in $[0,1]^d$, and for $p\in [1,\infty]$, let $\|\cdot \|_p$ denote the $\ell_p$ norm. 
Define the \emph{covering radius} of $P$ (in $\ell_p$ norm) by
\[ h_p(P):=\sup_{\bsx\in [0,1]^d}\min_{\bsy\in P} \|\bsx-\bsy\|_p, \]
the \emph{separation radius} of $P$ (in $\ell_p$ norm) by
\[ q_p(P):=\min_{\substack{\bsx,\bsy\in P\\ \bsx\neq \bsy}} \frac{\|\bsx-\bsy\|_p}{2},\]
and the \emph{mesh ratio} of $P$ (in $\ell_p$ norm) by
\begin{equation*}
\rho_p(P) = \frac{h_p(P)}{q_p(P)}.
\end{equation*}
It is well known that $\rho_p(P) \ge 1$ holds true since $[0,1]^d$ is connected.
As stated in Appendix~A of Part~I \cite{DGLPSxx}, these radii satisfy the bounds
\begin{align*}
h_p(P) &\ge \frac{1}{N^{1/d}}\cdot \frac{(\Gamma(1+d/p))^{1/d}}{2\Gamma(1+1/p)},\\
q_p(P) &\le \frac{(\Gamma(1+d/p))^{1/d}}{2N^{1/d}\Gamma(1+1/p)-2(\Gamma(1+d/p))^{1/d}}.
\end{align*} 
Consequently, the mesh ratio $\rho_p(P)=h_p(P)/q_p(P)$ is bounded whenever $h_p(P)=O(N^{-1/d})$ and $q_p(P)=\Theta(N^{-1/d})$.

Then quasi-uniformity for the set of point sets and for the sequence of points is defined as follows:
\begin{definition}\label{def_qu_pointsets}
    Let $X$ be a set of point sets in $\Omega$ with $\sup_{P \in X} |P| = \infty$.
    Then we call $X$ \emph{a quasi-uniform family} if there exists a constant $C_{p} >0$ such that the mesh ratio is bounded, i.e.
    \[ \rho_p(P) \leq C_{p}, \quad \forall P  \in X.\]
\end{definition}

\begin{definition}\label{def:qu-points}
Let $\Scal = (\bsx_n)_{n \ge 0}$ be an infinite sequence of points. We denote the set of the first $i$ points of $\Scal$ by $P_i =\{\bsx_0, \bsx_1, \ldots, \bsx_{i-1}\}$. Then, we call $\Scal$ a \emph{quasi-uniform sequence} if the set $\{P_i\}_{i \in \NN}$ is a quasi-uniform family, i.e., there is a constant $C'_p > 0$ such that \[ \rho_p(P_i) \le C'_p, \quad \forall i \in \NN.\]
We call the sequence $\Scal$ \emph{quasi-uniform along the subsequence} $1 \le i_1 < i_2 < \dots$ if the mesh ratio is bounded for all $i_k$, i.e., $\rho_p(P_{i_k}) \le C'_p$ for all $k$.
\end{definition}

Note that one can choose 
$C_p=2$ independently of 
$d$ \cite{PZ23}.
As stated in Part I, in our setting the constants usually depend on 
$d$. We leave the dimension dependence of the mesh ratio for QMC point sets for future work.

Since $\bsx, \bsy \in [0,1]^d$ are finite dimensional vectors, all norms $\| \cdot \|$ are equivalent, hence if a sequence of point sets is quasi-uniform for one $p$, it is quasi-uniform for all $p \in [1,\infty]$.

\subsection{Some known results for digital nets and sequences}
We have already introduced the overall literature and results in Part I, so here we will present those related to digital nets.

We have already stated in Part~1 (without proof) that the van der Corput sequence in base $b$ is not only of low discrepancy but also quasi-uniform. We will give a brief proof for the quasi-uniformity. 
The van der Corput sequence in base $b$ is the digital sequence 
whose generating matrix is the $\infty \times \infty$ identity matrix.
Recalling that the set of the first $b^m$ points is given by $\{i/b^m\mid i=0,\ldots,b^m-1\}$, for any $b^{m-1}<i\le b^m$,
the first $i$ points of the van der Corput sequence, denoted by $Q_i$, satisfy
\[ h_p(Q_i)\le h_p(Q_{b^{m-1}})=\frac{1}{ b^{m-1}} \quad \text{and}\quad q_p(Q_i)\ge q_p(Q_{b^{m}})=\frac{1}{2b^{m}}. \]
Therefore, the mesh ratio is bounded independently of $i$ as
\[ \rho_p(Q_i) = \frac{h_p(Q_i)}{q_p(Q_i)}\le 2b. \]
This means that the van der Corput sequence is quasi-uniform. 
With the additional point at $x=1$, the mesh ratio is uniformly bounded by $b$, which proves that the optimal mesh ratio of $2$, shown in \cite[Theorem~1.1]{PZ23}, can be achieved in the case $b=2$.

We have stated in Part I that the covering radius has been well studied under the name of \emph{dispersion} \cite[Chapter~6]{N92}.
In particular, any $(t,m,d)$-net $Q_{b^m}$ in base $b$, or the first $b^m$ points from any $(t,d)$-sequence $Q_{b^m}$ in base $b$, satisfy the following upper bound on the covering radius:
\[ h_{\infty}(Q_{b^m})\leq b^{(d+t-1)/d}b^{-m/d}, \]
which is of optimal order. For a digital sequence, \cite[Lemma~1.3]{DGLPSxx}
shows that the first $N$ points $Q_N$ of the digital sequence for any $N \ge 2$ satisfies $h_\infty(Q_N) \le b^{(d+t)/d} N^{-1/d}$.

For $d=2$, the separation radius of several $(0,m,2)$-nets was studied in \cite{GHSK08,GK09}, with particular interest in the \emph{toroidal distance}:
\[ \|\bsx-\bsy\|_{p,T}:=\left(\sum_{j=1}^{d}\left( \min(|x_j-y_j|,1-|x_j-y_j|)\right)^p\right)^{1/p},\]
for $p\in [1,\infty)$, with the obvious modification for $p=\infty$.
Numerical experiments therein suggest that the sequence of the Larcher--Pillichshammer digital $(0,m,2)$-nets from \cite{LP01} is quasi-uniform.
In this paper, we theoretically prove a stronger result: any digital shift preserves the quasi-uniformity of these nets (with respect to the toroidal distance).
In what follows, when we work with the covering radius, separation radius, and mesh ratio with respect to the toroidal distance, we denote them by $h_{p,T}, q_{p,T}$, and $\rho_{p,T}$, respectively. 
Note that, from the definition of the toroidal distance, both $h_{p,T}(P)\le h_{p}(P)$ and $q_{p,T}(P)\le q_{p}(P)$ hold for any $P\subset [0,1]^d$.

\section{Well-separated nets}\label{sec:well-separated}

\subsection{Basic concepts and properties}

We introduce the concept of $(t,m,d)$-nets in base $b$.

\begin{definition}
    Let $b\geq 2$ be an integer. A $d$-dimensional, $b$-adic elementary interval is an interval of the form
    \[ J_{\bsc,\bsa}:=\prod_{j=1}^{d}\left[ \frac{a_j}{b^{c_j}}, \frac{a_j+1}{b^{c_j}}\right)\]
    with integers $0\leq a_j<b^{c_j}$ and $c_j\geq 0$ for all $j$.
\end{definition}
A set $Q\subset [0,1)^d$ with $N=b^m$ points is called $(t,m,d)$-net in base $b$ if exactly $b^t$ points of $Q$ exist in each of $b^{m-t}$ elementary intervals for any choice of $c_1,\ldots,c_d$ satisfying $c_1+\cdots+c_d=m-t$.

Let $Q$ be a $(t,m,d)$-net in base $b$. 
As already mentioned in the previous section, it has been known from \cite[Theorem~6.10]{N92} that 
\begin{equation}
h_{\infty,T}(Q)\le h_{\infty}(Q) \le \frac{b^{(d+t-1)/d}}{b^{m/d}}, \label{eq:CoveringRadius-tmsnet}    
\end{equation}
implying that $(t,m,d)$-nets in base $b$ achieve the optimal order of the covering radius. To discuss the separation radius, we extend the notion of elementary intervals.

\begin{definition}
Let $b\geq 2$ be an integer. A shifted $b$-adic elementary interval is an interval of the form
    \[ J_{\bsc,\bsa, \bse}:=\prod_{j=1}^{d}\left[ \frac{a_j-e_j/b}{b^{c_j}}, \frac{a_j+(b-e_j)/b}{b^{c_j}}\right)\]
with integers $0\leq a_j \le b^{c_j}$, $c_j\geq 0$ and $e_j \in \{0, 1, 2, \ldots, b-1\}$ for all $j$.

Furthermore, 
a shifted toroidal $b$-adic elementary interval is an interval of the form
    \[ \tilde{J}_{\bsc,\bsa, \bse}:=\prod_{j=1}^{d}\left( \left[ \frac{a_j-e_j/b}{b^{c_j}}, \frac{a_j+(b-e_j)/b}{b^{c_j}}\right) \bmod 1 \right) \]
with integers $0\leq a_j < b^{c_j}$, $c_j\geq 0$ and $e_j \in \{0, 1, 2, \ldots, b-1\}$ for all $j$,
where, for a set $A \subset \RR$ we define \(
A \bmod 1 := \{x \in [0,1) \mid \text{there exists a $z \in \ZZ$ s.t. } x+z \in A \}.
\)
\end{definition}

Note that $\tilde{J}_{c,a,e} = J_{c,a,e}$
unless $a = 0$ and $e = \{1,\dots, b-1\}$.
We also note that $J_{\bsc,\bsa,\bse}$ is not necessarily contained in $[0,1)^d$ but $\tilde{J}_{\bsc,\bsa,\bse}$ is.
We now introduce the concept of $\bsc$-separated and $\kappa$-separated point sets in base $b$.

\begin{definition}\label{def_well_sep_net}
Let $P \subset [0,1)^d$ be a finite point set.
We call $P$ a $\bsc$-separated point set in base $b$
if every shifted $b$-adic elementary interval $J_{\bsc,\bsa,\bse}$
with $0 \le a_j \le b^{c_j}$ and $e_j \in \{0,1,\ldots,b-1\}$ for $j=1,\ldots,d$
contains at most one point of $P$.
If $P$ is a $(\kappa,\ldots,\kappa)$-separated point set,
then we call $P$ a {\it $\kappa$-separated point set} in base $b$.

Furthermore, we call a point set $P \subset [0,1)^d$ a
toroidally $\bsc$-separated point set in base $b$
if every toroidal shifted $b$-adic elementary interval $\tilde{J}_{\bsc,\bsa,\bse} \subseteq [0,1)^d$
with $0 \le a_j < b^{c_j}$ and
$e_j \in \{0,1,\ldots,b-1\}$ for $j=1,\ldots,d$
contains at most one point of $P$.
If $P$ is a toroidally $(\kappa,\ldots,\kappa)$-separated point set,
then we call $P$ a toroidally $\kappa$-separated point set in base $b$.
\end{definition}

The following lemma is to connect $\bsc$-separated and $\kappa$-separated.

\begin{lemma}\label{lem:c2kappa}
Let $P \subset [0,1)^d$ be a finite point set.
Let $\bsc,\bsc',\bsc''$ be three distinct vectors in $\NN_0^d$ such that $c'_j \le c_j \le c''_j$ for all $j$.
Then the following hold:
\begin{enumerate}
    \item If $P$ is not (toroidally or non-toroidally) $\bsc$-separated, then $P$ is not $\bsc'$-separated, and in particular, not $\kappa$-separated with $\kappa = \min(c_1,\ldots,c_d)$. \label{item1:c2kappa}
    \item If $P$ is $\bsc$-seaprated, then $P$ is $\bsc''$-separated, and in particular, $\kappa$-separated with $\kappa = \max(c_1,\ldots,c_d)$.
    \label{item2:c2kappa}
\end{enumerate}
\end{lemma}

\begin{proof}
First, for any admissible parameters $\bsa,\bse$, we explicitly construct
$\bsa',\bse'$ such that
$J_{\bsc,\bsa,\bse} \subset J_{\bsc',\bsa',\bse'}.$
Indeed, for $j \in \{1,\ldots,d\}$, define
\[
m_j := \left\lfloor \frac{b a_j - e_j }{b^{c_j-c'_j}} \right\rfloor
      \in \{0,1,\ldots,b^{c'_j+1}-1\},
\]
and write $m_j = b a'_j - e'_j$ with $0 \le e'_j \le b-1$. This choice of $\bsa',\bse'$, satisfies $J_{\bsc,\bsa,\bse} \subset J_{\bsc',\bsa',\bse'}$.
The same construction applies verbatim to the toroidal intervals $\tilde J_{\bsc,\bsa,\bse}$.

We now show \eqref{item1:c2kappa}.
we assume $P$ is not $\bsc$-separated, i.e., there exists some $\bsa$ and $\bse$ such that
$J_{\bsc,\bsa,\bse}$ contains at least two points from $P$. We take $\bsa'$ and $\bse'$ as above such that $J_{\bsc,\bsa,\bse} \subset J_{\bsc',\bsa',\bse'}$.
Since $J_{\bsc',\bsa',\bse'}$ also contains at least two points from $P$, $P$ is not $\bsc'$-separated.

\eqref{item2:c2kappa} is the contrapositive of
\eqref{item1:c2kappa}.
\end{proof}

For a fixed point set, a smaller value of $\kappa$ indicates better quality.
The next lemma shows that, for $\kappa$-separated point sets, all the points are at least distance $(b-1)b^{-\kappa}/(2b)$ apart in the $\ell_{\infty}$ norm, which motivates the name ``$\kappa$-separated''.

\begin{lemma}\label{lem_well_sep}
Let $P \subset [0,1)^d$ be a $\kappa$-separated point set in base $b$. Then we have
\begin{equation*}
q_\infty(P) \ge \frac{b-1}{2b} b^{-\kappa}.
\end{equation*}
\end{lemma}

\begin{proof}
Let $\bsx = (x_1,\dots,x_d) \in P$.  
Let the base $b$ representation of $x_j$ be given by
\begin{equation*}
x_j = \frac{x_{j,1}}{b} + \frac{x_{j,2}}{b^2} + \cdots,
\end{equation*}
unique in the sense that infinitely many of the $x_{j,k}$ are different from $b-1$. 
For each $j=1,\dots,d$, we define
\begin{align}
x^*_j &:= \frac{x_{j,1}}{b} + \dots + \frac{x_{j,\kappa+1}}{b^{\kappa+1}}, \notag\\
I_{j,1} &:= [x^*_j - (b-1)/b^{\kappa+1}, x^*_j + 1/b^{\kappa+1}) = J_{\kappa,a_{j,1},e_{j,1}} \label{eq:I1}\\
I_{j,2} &:= [x^*_j,x^*_j + 1/b^{\kappa}) = J_{\kappa,a_{j,2},e_{j,2}}, \label{eq:I2}
\end{align}
where $a_{j,1},e_{j,1},a_{j,2},e_{j,2}$ satisfies
$ba_{j,1}-e_{j,1} = b^{\kappa+1}x_j^*-(b-1)$ and 
$ba_{j,2}-e_{j,2} = b^{\kappa+1}x_j^*$.
Since $x_j^* \in I_{j,k}$ holds for any $1 \le j \le d$ and $k \in \{1,2\}$,
for any $\bsk = (k_1,\dots,k_d) \in \{1,2\}^d$ we have
\[
\bsx \in \prod_{j=1}^d I_{j,k_j} =: I_{\bsk}.
\]
From \eqref{eq:I1} and \eqref{eq:I2},
$I_{\bsk}$ is a shifted $b$-adic elementary interval of the form $J_{(\kappa,\dots,\kappa),\bsa,\bse}$
for any $\bsk \in \{1,2\}^d$.
Since $P$ is $\kappa$-separated,
for every $\bsy \in P \setminus \{\bsx\}$ we have $\bsy \notin I_{\bsk}$,
and thus 
\[
\bsy \notin \bigcup_{\bsk \in \{1,2\}^d} I_{\bsk}
=  \prod_{j=1}^d[x^*_j - (b-1)/b^{\kappa+1}, x^*_j + 1/b^{\kappa}).
\]
Thus, noting that $x_j^*\le x_j < x_j^*+1/b^{\kappa+1}$ holds for all $j$, we have $\max_{1 \le j \le d}|x_j - y_j| \ge (b-1)/b^{\kappa+1}$
and the result follows.
\end{proof}

Similarly, we can show that the separation radius of toroidally $\kappa$-separated point sets remains large (with respect to the toroidal distance) even when any shift is applied to these point sets.
Since the proof goes in the same way, we omit the proof.
\begin{lemma}\label{lem_well_sep_shifted}
Let $P \subset [0,1)^d$ be a toroidally $\kappa$-separated point set.
Let $\bsdelta \in [0,1)^d$ and 
$P+\bsdelta := \{\bsx+\bsdelta \bmod{1} \mid \bsx \in P\}$ be the shifted point set in $[0,1)^d$.
Then, for any $\bsdelta$, we have
\begin{equation*}
q_{\infty,T}(P+\bsdelta) \ge \frac{b-1}{2b} b^{-\kappa}.
\end{equation*}
\end{lemma}

We can also prove an inverse of the previous lemmas.

\begin{lemma}\label{lem:nonwell_is_smallsep}
Let $P \subset [0,1)^d$ be a finite point set. Let $\kappa \ge 0$ and assume that $P$ is not $\kappa$-separated. Then
\[ q_\infty(P) \le \frac{1}{2 b^{\kappa}}.\]
Similarly, assume that $P$ is not toroidally $\kappa$-separated. Then
\[ q_{\infty,T}(P) \le \frac{1}{2 b^{\kappa}}.\]
\end{lemma}
\begin{proof}
We provide a proof for the first claim only, as the second claim can be proven in a similar manner.
Let $\bsc = (\kappa, \kappa, \ldots, \kappa)$. Since $P$ is not $\kappa$-separated, there exist $a_1, a_2 \ldots, a_d \in \mathbb{Z}$ such that $0 \le a_1, a_2, \ldots, a_d \le b^{\kappa}$ and $e_1, e_2, \ldots, e_d \in \{0, 1, \ldots, b-1\}$ such that $J_{\bsc, \bsa, \bse} \subseteq [0,1]^d$ and $|J_{\bsc,\bsa,\bse} \cap P | > 1$. Assume that $\bsx, \bsy \in P \cap J_{\bsc,\bsa,\bse}$. Then
\begin{equation*}
\|\bsx -\bsy\|_\infty = \max_{1 \le j \le d} |x_j - y_j| \le \max_{1 \le j \le d} \frac{a_j + (b-e_j)/b - a_j + e_j/b}{b^\kappa} = b^{-\kappa}.
\end{equation*}
Thus
\begin{equation*}
q_\infty(P) \le \frac{1}{2 b^{\kappa}}.
\qedhere
\end{equation*}
\end{proof}

Let $\{Q_{b^m}\}_{m \in \NN}$ be a set of point sets, where each $Q_{b^m}$ is a $(t,m,d)$-net in base $b$. 
For $\{Q_{b^m}\}_{m \in \NN}$ to be a quasi-uniform family, the separation radius $q_{\infty}(Q_{b^m})$ (resp. $q_{\infty,T}(Q_{b^m})$) should be bounded below by the covering radius $h_{\infty}(Q_{b^m})$ (resp. $h_{\infty,T}(Q_{b^m})$) up to a constant factor.
Given the upper bound on $h_{\infty,T}(Q_{b^m})$ and $h_{\infty}(Q_{b^m})$ shown in \eqref{eq:CoveringRadius-tmsnet}, $Q_{b^m}$ needs to be a $\kappa$-separated point set with $\kappa\approx m/d$.
This motivates the following definition of well-separated families and sequences.

\begin{definition}
Let $I \subset \NN$ be an infinite set and
$\{Q_{b^m}\}_{m \in I}$ be a family of point sets indexed by $I$ such that $Q_{b^m}$ is a $(t,m,d)$-net in base $b$. 
If for all $m \in I$, $Q_{b^m}$ is $\kappa_m$-separated with $\kappa_m \le m/d+O(1)$, then we call $(Q_{b^m})_{m \in \NN}$ \emph{well-separated}, or a \emph{well-separated family on $I$}.
A sequence of points $\Scal$ is said to be well-separated, or a well-separated sequence, if the set
$\{Q_{b^m}\}_{m \in \NN}$ of its initial point sets is a well-separated family on $\NN$, where
$Q_{b^m}$ denotes the first $b^m$ points of $\Scal$.

We define the term ``toroidally well-separated'' in the same way.
\end{definition}

If $(Q_{b^m})_{m \in I}$ is a well-separated family on certain index set $I$,
it follows from \eqref{eq:CoveringRadius-tmsnet} and Lemma~\ref{lem_well_sep} that
\[ \rho_\infty(Q_{b^m})=\frac{h_\infty(Q_{b^m})}{q_\infty(Q_{b^m})}\leq \frac{b^{(d+t-1)/d}}{b^{m/d}}\cdot \frac{2b^{\kappa_m+1}}{b-1}\leq \frac{2b^{(t-1)/d+O(1)}}{b-1}, \]
for all $m\in I$. Thus,
by using \cite[Lemma~1.3]{DGLPSxx}
 in the quasi-uniform case,
the following corollary holds:

\begin{corollary}
    The following holds true:
    \begin{enumerate}
    \item 
    Any well-separated family $\{Q_{b^m}\}_{m \in I}$ is a quasi-uniform family. 

    \item
    Any well-separated sequence $\Scal = \{\bsx_0,\bsx_1,\dots\}$ is a quasi-uniform sequence.    
    \end{enumerate}
\end{corollary}

Since $\rho_{\infty,T}(Q_{b^m})$ has the same upper bound as shown above for a toroidally well-separated sequence, we have an equivalent statement about the quasi-uniformity with respect to the toroidal distance. 

\subsection{An algebraic criterion for \texorpdfstring{$\bsc$}{c}-separated digital nets}\label{subsec_well_separated}

For generating matrices $F_1,\ldots,F_d\in \FF_b^{m\times m}$, let us write
\begin{equation*}
F_j = \begin{pmatrix} f_{j,1} \\ f_{j,2} \\ \vdots \\ f_{j,m} \end{pmatrix}\quad \text{with $f_{j,1},\ldots,f_{j,m}\in \FF_b^m$}.
\end{equation*}
If for all $c_j \ge 0$ such that $c_1 + c_2 + \cdots + c_d = m-t$, the vectors $$f_{1,1}, \ldots, f_{1, c_1}, f_{2,1}, \ldots, f_{2, c_2}, \ldots, f_{d, 1}, \ldots, f_{d, c_d}$$ are linearly independent over $\FF_b$, then it is known that, for any $\bsdelta \in \{0, 1/b^m, 2/b^m, \ldots, (b^m-1)/b^m\}^d$, the digitally shifted digital net $Q(F_1, \ldots, F_d)\oplus \bsdelta$ is a $(t,m,d)$-net in base $b$, see, e.g., \cite[Theorem~4.52 \& Lemma 4.67]{DP10}, and we call it a (digitally shifted) digital $(t,m,d)$-net over $\FF_b$.

\bigskip

In the following, we give a condition on the generating matrices and digital shift such that the corresponding digital net is (toroidally) $\kappa$-separated.

Let $Q(F_1, \ldots, F_d)$ be a digital net over $\FF_b$, where $F_1, \dots, F_d \in \FF_b^{m\times m}$.
Let $\bsdelta = (\delta_1, \delta_2, \ldots, \delta_d) \in \{0, 1/b^m, 2/b^m, \ldots, (b^m-1)/b^m\}^d$
and $\delta_j = \delta_{j,1} b^{-1} + \cdots + \delta_{j,m} b^{-m}$ with $\delta_{j,k} \in \{0, 1, \ldots, b-1\}$. Define $\vec{\delta}_j = (\delta_{j,1}, \delta_{j,2}, \ldots, \delta_{j, c_j}, \delta_{j,c_j+1})^\top$. 
Our point set is the digitally shifted digital net $Q(F_1, \ldots, F_d) \oplus \bsdelta$.
For our purpose, for given $\bsc, \bsa, \bse$,
we find the number of points contained in
$\tilde{J}_{\bsc,\bsa,\bse}$.
Since we have
\begin{equation}\label{eq_J_union}
\tilde{J}_{\bsc,\bsa,\bse} = \bigcup_{\bsg \in \{0,1,\dots,b-1\}^d}  J_{\bsc+\bsone,(b\times\bsa-\bse+\bsg) \bmod b^{\bsc+\bsone}}
\end{equation}
where
\begin{align*}
\bsc + \bsone &:= (c_j+1)_{j=1}^d,\\
(b\times\bsa-\bse+\bsg) \bmod b^{\bsc+\bsone} &:= ((ba_j-e_j+g_j) \bmod b^{c^j+1})_{j=1}^d,
\end{align*}
and all $J_{\bsc+\bsone,(b\times\bsa-\bse+\bsg) \bmod b^{\bsc+\bsone}}$ are disjoint with each other,
it suffices to find the number of points in the elementary interval $J_{\bsc+\boldsymbol{1},(b\times\bsa-\bse+\bsg) \bmod b^{\bsc+\bsone}}$ for each $\bsg = (g_1,\dots,g_d) \in \{0,1,\dots,b-1\}^d$.
For $g_j = 0, 1, \ldots, b-1$, let
\begin{equation*}
(ba_{j} - e_j + g_j) \bmod{b^{c_j+1}} = b^{c_j} A^{(g_j)}_{j,1} + {b^{c_j-1}}A^{(g_j)}_{j,2} + \cdots + b A^{(g_j)}_{j, c_j} + A^{(g_j)}_{j, c_j+1}
\end{equation*}
where the coefficients $A_{j,k}^{(g_j)} \in \{0, 1, \ldots, b-1\}$.
Since $|e_j-g_j| \le b-1$ holds,
for fixed $\bsc,\bsa,\bse$, we have
\begin{align}
b^{c_j-1}A^{(g_j)}_{j,1} &+ b^{c_j-2}A^{(g_j)}_{j,2} + \cdots + A^{(g_j)}_{j, c_j} \notag\\
&= \begin{cases}
a_j & \text{if } g_j \ge e_j, \\
(a_j-1) \bmod b^{c_j} & \text{if } 0 \le g_j < e_j,
\end{cases} \notag\\
&= \begin{cases}
b^{c_j-1}A^{(e_j)}_{j,1} + b^{c_j-2}A^{(e_j)}_{j,2} + \cdots + A^{(e_j)}_{j, c_j} & \text{if } g_j \ge e_j, \\
b^{c_j-1}A^{(0)}_{j,1} + b^{c_j-2}A^{(0)}_{j,2} + \cdots + A^{(0)}_{j, c_j} & \text{if } 0 \le g_j < e_j.
\end{cases} \label{eq_A_change}
\end{align}
In particular, for $e_j \neq 0$ we have
\begin{equation}\label{eq_coeff_change}
A_{j,c_j}^{(e_j)} = 1 + A_{j, c_j}^{(0)} \pmod{b}.
\end{equation}

For $j = 1, 2, \ldots, d$, define the vectors 
$\vec{A}_j^{(g_j)} \in \FF_b^{c_j+1}$ by
\begin{align*}
\vec{A}_j^{(g_j)} &= \begin{pmatrix} A^{(g_j)}_{j,1} \\ \vdots \\ A^{(g_j)}_{j, c_j} \\ A^{(g_j)}_{j, c_j+1} \end{pmatrix} \\
&= \begin{cases} (A_{j,1}^{(0)}, \ldots, A_{j,c_j}^{(0)}, g_j-e_j)^\top & \mbox{ if } e_j=0, g_j \in \{0, 1, \ldots, b-1\}, \\
(A^{(0)}_{j,1}, \ldots, A^{(0)}_{j,c_j},  g_j-e_j)^\top & \mbox{ if } e_j \neq 0, g_j \in \{0, 1, \ldots, e_j-1\}, \\
(A^{(e_j)}_{j,1}, \ldots, A^{(e_j)}_{j,c_j}, g_j - e_j)^\top & \mbox{ if } e_j \neq 0, g_j \in \{e_j, e_j+1, \ldots, b-1\}.
\end{cases}
\end{align*}
and the matrices
\begin{equation*}
F_{j,c_j+1} = \begin{pmatrix} f_{j,1} \\ f_{j,2} \\ \vdots \\ f_{j, c_j + 1} \end{pmatrix},
\end{equation*}
where we formally regard 
$f_{j,m+1} = \bszero$ if $c_j = m$, for $1 \le j \le d$.
Then $\bsx_n \oplus \bsdelta \in J_{\bsc+\bsone,(b\times\bsa-\bse+\bsg) \bmod b^{\bsc+\bsone}}$
if and only if, for all $j = 1, 2, \ldots, d$, 
it holds that
\begin{equation}\label{dig_net_lin_systems}
F_{j, c_j+1} \vec{n}  = \vec{A}^{(g_j)}_j - \vec{\delta}_j \in \mathbb{F}_b^{c_{j}+1}.
\end{equation}

Let $F_{\bsc+\bsone}\in \FF_b^{(c_1+\dots+c_d+d) \times m}$ and $\vec{B}_{\bsc+\bsone, \bsdelta, \bsg} \in \FF_b^{(c_1+\dots+c_d+d) \times 1}$ be
\begin{equation}\label{eq:FB}
F_{\bsc+\bsone} = \begin{pmatrix} F_{1,c_1 +1} \\ F_{2, c_2+1} \\ \vdots \\ F_{d, c_d+1} \end{pmatrix} 
\quad \mbox{ and } \quad
\vec{B}_{\bsc+\bsone, \bsdelta, \bsg} = \begin{pmatrix} \vec{A}^{(g_1)}_1 - \vec{\delta}_1 \\  \vdots \\ \vec{A}^{(g_d)}_d - \vec{\delta}_d \end{pmatrix}.
\end{equation}
Then, from the above argument,
the number of points in $J_{\bsc+\bsone,(b\times\bsa-\bse+\bsg) \bmod b^{\bsc+\bsone}}$ is given by the number of solutions $\vec{n}$ of the linear systems
\begin{equation}\label{eq_lin_sys_counting}
F_{\bsc+\bsone} \vec{n} = \vec{B}_{\bsc+\bsone, \bsdelta, \bsg}.
\end{equation}
Using \eqref{eq_J_union} it follows that the number of points in $\tilde{J}_{\bsc,\bsa,\bse}$ is the sum of the number of solutions of \eqref{eq_lin_sys_counting}
where $\bsg$ runs through all admissible choices $\{0,1,\dots,b-1\}^d$.
It is also worth noting that the number of points in $J_{\bsc,\bsa,\bse}$ is bounded by the number of points in $\tilde{J}_{\bsc,\bsa,\bse}$.

In order to prove a bound on the minimum distance using Lemma~\ref{lem_well_sep}, we need to show that the digital net is $\bsc$-separated for some $\bsc$ as defined in Definition~\ref{def_well_sep_net}. This means that we need to show that there is at most one point in $\tilde{J}_{\bsc,\bsa,\bse}$. In terms of \eqref{eq_lin_sys_counting}, this means that the set
\begin{equation*}
\mathcal{N} = \left\{ \vec{n} \in \mathbb{F}_b^m: \exists \bsg \in \{0, 1, \ldots, b-1\}^d \mbox{ such that } F_{\bsc+\bsone} \vec{n} = \vec{B}_{\bsc+\bsone, \bsdelta, \bsg} \right\}
\end{equation*}
has at most one element.

We now further assume that the row rank of $F_{\bsc+\bsone}$ is equal to $m$.
Note that this is satisfied if $c_1+\cdots + c_m + d \ge m$ and $Q$ is a $(0,m,d)$-net.
Under this assumption, for any given $\bsg \in \{0, 1, \ldots, b-1\}^d$, the number of solutions of \eqref{eq_lin_sys_counting} is at most one. 

To show that $\mathcal{N}$ can have at most one solution, we proceed in the following way. Assume that $c_1 + \cdots + c_d + d > m$. Then the rows of the matrix $F_{\bsc+\bsone}$ are linearly dependent. If for a given $\bsg$, the augmented matrix $(F_{\bsc+\bsone} \mid \vec{B}_{\bsc+\bsone, \bsdelta, \bsg})$ does not satisfy the same linear dependence relation as $F_{\bsc+ \bsone}$, then \eqref{eq_lin_sys_counting} cannot have a solution $\vec{n}$, because the equations are not consistent. Thus, if there is at most one element $\bsg \in \{0, 1, \ldots, b-1\}^d$ such that $(F_{\bsc+\bsone} \mid \vec{B}_{\bsc+\bsone, \bsdelta, \bsg})$ satisfies the same linear dependence relations as $F_{\bsc+\bsone}$, then $\mathcal{N}$ has at most one solution, which implies that the digital net is $\bsc$-separated.
Noting that $F_{\bsc+\bsone}$ is independent of $\bsa,\bse,\bsg$,
we can precompute such independent equations. Further note that $\vec{B}_{\bsc+\bsone, \bsdelta, \bsg}$ does not depend $\mathbb{F}_b$ linearly on $g_j$ when $e_j \neq 0$.

The above argument can be summarized into the following algebraic criterion.
\begin{theorem}\label{thm:criterion}
Let $m \in \NN$, $\bsdelta \in \{0, 1/b^m, 2/b^m, \ldots, (b^m-1)/b^m\}^d$
and $Q(F_1, \ldots, F_d)$ be a digital net over $\FF_b$, where $F_1, \dots, F_d \in \FF_b^{m\times m}$.
Let $\bsc=(c_1,\dots,c_d) \in \NN_0^d$ such that $c_j \le m$ for $1 \le j \le d$ and $c_1 + \cdots + c_d + d > m$, and
assume that the row rank of $F_{\bsc+\bsone}$, given in \eqref{eq:FB}, is equal to $m$.
Here we formally regard in \eqref{eq:FB} that
$f_{j,m+1} = \bszero$ if $c_j = m$, for $1 \le j \le d$.
Then, for any $\bsa \in \prod_{j=1}^d\{0,1,\dots,b^{c_j}-1\}$ and $\bse \in \{0,1,\dots,b-1\}^d$,
the following are equivalent:
\begin{enumerate}
\item 
In the shifted toroidal elementary interval $\tilde{J}_{\bsc,\bsa,\bse}$,
there exist at most one point of the digitally shifted digital net $Q(F_1, \ldots, F_d) \oplus \bsdelta$. 
\item 
The sum of the number of solutions of \eqref{eq_lin_sys_counting},
with respect to $\bsg$ ranging over $\{0,1,\dots,b-1\}^d$,
is at most one.

\item 
There exist at most one $\bsg \in \{0, 1, \ldots, b-1\}$ such that,
for any nontrivial 
$\FF_b$-linear dependence relations of row vectors of 
$F_{\bsc+\bsone}$, 
the same relations are also satisfied for the same rows of the augmented matrix $(F_{\bsc+\bsone} \mid \vec{B}_{\bsc+\bsone, \bsdelta, \bsg}).$
\end{enumerate}
\end{theorem}

Using this theorem, we can check if a given digitally shifted digital net is $\bsc$-separated, either toroidally or not.
That is, check if the third condition of the lemma holds true for any choice of $\bsa$ and $\bse$.

\subsection{An example of two-dimensional well-separated digital nets}

Based on the previous section, we now show that the Larcher--Pillichshammer digital nets (LP net) over $\FF_b$ from \cite{LP01} are well-separated. The LP net has $m \times m$ generating matrices
\begin{equation*}
F_1 = \begin{pmatrix} 
1 & 0 & \ldots & \ldots & 0 \\
0 & 1 & 0      & \ldots & 0 \\
\vdots & \ddots & \ddots & \ddots & \vdots \\
0 & \ldots & 0 & 1 & 0 \\
0 & \ldots & \ldots & 0 & 1
\end{pmatrix} 
\quad\mbox{ and }\quad F_2 = \begin{pmatrix}
1 & 1 & \ldots & \ldots & 1 \\
1 & 1 & \ldots & 1 & 0 \\
\vdots & \vdots & \iddots & \iddots & \vdots \\
1 & 1 & 0 & \ldots & 0 \\
1 & 0 & \ldots & 0 & 0
\end{pmatrix}.
\end{equation*}
It is well known that the LP net is a $(0,m,2)$-net in base $b$.

\begin{theorem}\label{thm_LP_quasi_uniform}
Let $m\geq 1$ be an integer. 
For any $\bsdelta \in \{0, 1/b^m, 2/b^m, \ldots, (b^m-1)/b^m\}^2$, the digitally shifted Larcher--Pillichshammer net $Q(F_1, F_2) \oplus \bsdelta$ is $\lceil m/2 \rceil +1$-separated and also toroidally $\lceil m/2 \rceil + 1$-separated.
\end{theorem}

\begin{proof}
We first consider the case $1 \le m \le 3$, which implies that $m \le \lceil m/2 \rceil+1$.
Since any toroidally shifted elementary interval with $\bsc=(m,m)$ contains exactly one point in the set $\{0, 1/b^m, 2/b^m, \ldots, (b^m-1)/b^m\}^2$,
$Q(F_1, F_2) \oplus \bsdelta$ is toroidally $m$-separated and thus $\lceil m/2 \rceil+1$-separated.

We now assume $m \ge 4$.
Let $c_1 = \lceil m/2 \rceil+1$, $c_2 = \lfloor m/2 \rfloor + 1$.
We check the third condition in Theorem~\ref{thm:criterion} for each admissible choice of $\bsa, \bse$.
Let $f_{j,k} \in \FF_b^m$ denote the $k$-th row of $F_j$.
Since $m$ vectors consisting of the rows $f_{1,1}, \ldots, f_{1, c_1}$ together with $f_{2,1}, \ldots, f_{2, c_2-1}$ are linearly independent over $\FF_b$, the assumption of Theorem~\ref{thm:criterion} is satisfied.
We choose $3$ equations as
\begin{align*}
& f_{1,c_1}+f_{2,c_2} =  f_{2,c_2-1}, \\
& f_{1,c_1+1} = - f_{2,c_2-1} + f_{2,c_2-2}, \\
& f_{2,c_2+1} = f_{1,1} + \cdots + f_{1,c_1-2}. 
\end{align*}
Note that the assumption $m \ge 4$ assures that these equations are valid.
To check the third condition in Theorem~\ref{thm:criterion}, it suffices to show that at most one pair $(g_1, g_2)$ satisfies the following $3$ equations (which correspond to the equations for $f_{j,k}$ above):
\begin{align*}
    & A^{(g_1)}_{1,c_1}-\delta_{1,c_1} +  A^{(g_2)}_{2,c_2}-\delta_{2,c_2} =  A^{(g_2)}_{2,c_2-1}-\delta_{2,c_2-1} , \\
    & A^{(g_1)}_{1,c_1+1}-\delta_{1,c_1+1} = - (A^{(g_2)}_{2,c_2-1}-\delta_{2,c_2-1}) + (A^{(g_2)}_{2,c_2-2}-\delta_{2,c_2-2}),  \\
    & A^{(g_2)}_{2, c_2+1}-\delta_{2,c_2+1} = (A^{(g_1)}_{1,1}-\delta_{1,1}) + \cdots + (A^{(g_1)}_{1,c_1-2}-\delta_{1,c_1-2}). 
\end{align*}
In what follows, we consider four cases based on whether $e_1$ and $e_2$ are equal to $0$.

\begin{itemize}
    \item If $e_1 = e_2 = 0$, then there is at most one point in the corresponding interval due to the fact that the LP net is a $(0,m,2)$-net and so is the digitally shifted LP net. 

    \item If $e_1\neq 0$ and $e_2=0$, then 
    \begin{align*}
    & A^{(g_1)}_{1,c_1}-\delta_{1,c_1} + A^{(0)}_{2,c_2}-\delta_{2,c_2} =  A^{(0)}_{2,c_2-1}-\delta_{2,c_2-1}, \\
    & g_1-e_1-\delta_{1,c_1+1} = - (A^{(0)}_{2,c_2-1}-\delta_{2,c_2-1}) + (A^{(0)}_{2,c_2-2}-\delta_{2,c_2-2}),   \\
    & g_2-\delta_{2,c_2+1} = (A^{(g_1)}_{1,1}-\delta_{1,1}) + \cdots + (A^{(g_1)}_{1,c_1-2}-\delta_{1,c_1-2}).
    \end{align*}
    The second equation implies that there is only one choice for $g_1$. Then the third equation implies that there is only one choice for $g_2$.

    \item Let now $e_1 = 0$ and $e_2 \neq 0$. Then
    \begin{align*}
    & A^{(0)}_{1,c_1}-\delta_{1,c_1} + A^{(g_2)}_{2,c_2}-\delta_{2,c_2} =   A^{(g_2)}_{2,c_2-1}-\delta_{2,c_2-1}, \\
    & g_1-\delta_{1,c_1+1} = -(A^{(g_2)}_{2,c_2-1}-\delta_{2,c_2-1}) + (A^{(g_2)}_{2,c_2-2}-\delta_{2,c_2-2}),  \\
    & g_2-e_2-\delta_{2,c_2+1} = (A^{(0)}_{1,1}-\delta_{1,1}) + \cdots + (A^{(0)}_{1,c_1-2}-\delta_{1,c_1-2}). 
    \end{align*}
    The last equation implies that there is only one choice for $g_2$ and then the second equation implies that there is only one choice for $g_1$ as well.

    \item Finally let $e_1\neq 0$ and $e_2\neq 0$. Then 
    \begin{align}
    & A^{(g_1)}_{1,c_1}-\delta_{1,c_1} + A^{(g_2)}_{2,c_2}-\delta_{2,c_2} =   A^{(g_2)}_{2,c_2-1}-\delta_{2,c_2-1}, \label{lin_sys_eq1} \\
    & g_1-e_1-\delta_{1,c_1+1} = -(A^{(g_2)}_{2,c_2-1}-\delta_{2,c_2-1}) + (A^{(g_2)}_{2,c_2-2}-\delta_{2,c_2-2}),  \label{lin_sys_eq2} \\
    & g_2-e_2-\delta_{2,c_2+1} = (A^{(g_1)}_{1,1}-\delta_{1,1}) + \cdots + (A^{(g_1)}_{1,c_1-2}-\delta_{1,c_1-2}). \label{lin_sys_eq3}
    \end{align}

    Then \eqref{eq_coeff_change} implies that 
    \begin{equation}\label{LP_Eq1}
    A_{j,c_j}^{(g_j)} = A_{j, c_j}^{(0)} + \begin{cases} 0 & \mbox{if } 0 \le g_j < e_j, \\ 1 & \mbox{if } e_j \le g_j < b. \end{cases}
    \end{equation}

    We consider two cases. First assume that $A_{2, c_2-1}^{(0)} = A_{2, c_2-1}^{(e_2)}$, i.e. $A_{2, c_2-1}^{(g_2)}$ is independent of $g_2$. Then \eqref{eq_A_change} implies that $(a_2-1)+1$ has no carry over, i.e., $(a_2-1) \bmod{b} < b-1$, and thus $A_{2,k}^{(g_2)}$ is constant, independent of $g_2$, for all $1 \le k < c_2$. Then the second equation \eqref{lin_sys_eq2} implies that $g_1$ is constant, and in turn, the third equation \eqref{lin_sys_eq3} implies that $g_2$ is constant as well. In this case there is only one pair $(g_1, g_2)$ which satisfies all equations.

    Now assume that $A_{2, c_2-1}^{(g_2)}$ is not constant. Then 
    \begin{equation*}
    A_{2, c_2-1}^{(g_2)} = A_{2, c_2-1}^{(0)} + \begin{cases} 0 & \mbox{if } 0 \le g_2 < e_2, \\ 1 & \mbox{if }  e_2 \le g_2 < b. \end{cases}
    \end{equation*}
    Then \eqref{LP_Eq1} for $j=2$ and the first equation \eqref{lin_sys_eq1} implies that for all $g_2 \in \{0, 1, \ldots, b-1\}$, we get the same value for $A_{1,c_1}^{(g_1)}$. With \eqref{eq_coeff_change} and \eqref{eq_A_change}, it follows that $A_{1,c_1}^{(g_1)}$, whether $g_1<e_1$ or $g_1\ge e_1$, is determined independently of $g_2$. Then \eqref{eq_A_change} implies that $A_{1,k}^{(g_1)}$ 
    is also determined independently of $g_2$
    for all $1 \le k \le c_1$. Then the third equation \eqref{lin_sys_eq3} implies that $g_2$ is fixed. Then the second equation \eqref{lin_sys_eq2} implies that also $g_1$ is fixed and hence there is at most one possible choice for the pair $(g_1, g_2)$.
    \end{itemize}

    Therefore, for any admissible $\bsa,\bse$, there exists at most one point of the LP net in $\tilde{J}_{\bsc,\bsa,\bse}$, concluding that the LP net is $\kappa$-separated (either toroidally or not) with
    \[ \kappa=\max(c_1,c_2)=\lceil m/2 \rceil+1. \qedhere \]
\end{proof}

The proof of Theorem~\ref{thm_LP_quasi_uniform} can be generalized to other digital $(0,m,2)$-nets. We can change the generating matrices $F_1, F_2$ to $F'_1, F'_2$, provided that $F'_1, F'_2$ generate a digital $(0,m,2)$-net, and that the first $c_1+1$ rows of $F'_1$ are the same as the first $c_1+1$ rows of $F_1$ and analogously for $F'_2$. Then $F'_1, F'_2$ still generate $\lceil m/2 \rceil+1$-separated, digital $(0,m,2)$-net.

\section{Not well-separated digital nets}\label{sec:not-well-separated}

In this section, we provide some examples of digital $(0,m,d)$-nets that do not form a quasi-uniform family nor a well-separated family.
The proof strategy is to find two close points, specifically $\bsx_n$ and $\bsx_1$, or $\bsx_n$ and $\bsx_k$, which would imply a small separation radius.

\subsection{The Hammersley \texorpdfstring{$(0,m,2)$}{(0,m,2)}-net}

The generating matrices for the Hammersley net $\mathcal{H}_{m, b}$ over $\FF_b$ are given by
\begin{equation*}
F_1 = \begin{pmatrix} 
1 & 0 & \ldots & \ldots & 0 \\
0 & 1 & 0      & \ldots & 0 \\
\vdots & \ddots & \ddots & \ddots & \vdots \\
0 & \ldots & 0 & 1 & 0 \\
0 & \ldots & \ldots & 0 & 1
\end{pmatrix} \quad \mbox{ and }\quad  F_2 = \begin{pmatrix}
0 & \ldots & \ldots & 0 & 1 \\
0 & \ldots & 0 & 1 & 0 \\
\vdots & \iddots & \iddots & \iddots & \vdots \\
0 & 1 & 0 & \ldots & 0 \\
1 & 0 & \ldots & 0 & 0
\end{pmatrix}.
\end{equation*}

\begin{theorem}
The Hammersley $(0,m,2)$-nets do not form a quasi-uniform family nor a well-separated family.
Indeed we have
\begin{equation*}
q_\infty(\mathcal{H}_{m,b}) \le \frac{b-1}{2 \cdot b^m}.
\end{equation*}
\end{theorem}

\begin{proof}
Let $n = b^{m-1} + 1$ and $k = b^{m-1} - b$.
From their $b$-adic expansions, we have
\[ \vec{n}=(1,0,\ldots,0,1)^{\top}\quad\text{and}\quad \vec{k}=(0,b-1,\ldots, b-1, 0)^{\top}, \]
where both vectors have length $m$ as in Definition~\ref{def:digitalnet}.
Then
\begin{equation*}
\bsx_n = \left(\frac{1}{b} + \frac{1}{b^m}, \frac{1}{b} + \frac{1}{b^m} \right)
\end{equation*}
and
\begin{equation*} 
\bsx_k = \left(\frac{b-1}{b^2} + \cdots + \frac{b-1}{b^{m-1}}, \frac{b-1}{b^2} + \cdots + \frac{b-1}{b^{m-1}} \right) = \left(\frac{1}{b} - \frac{1}{b^{m-1}}, \frac{1}{b} - \frac{1}{b^{m-1}} \right).
\end{equation*}
Thus
\begin{equation*}
\bsx_n - \bsx_k = \left(\frac{b-1}{b^{m}}, \frac{b-1}{b^m} \right),
\end{equation*}
which shows that 
\begin{equation*}
q_\infty(\mathcal{H}_{m,b}) \le \frac{b-1}{2 \cdot b^{m}}.
\end{equation*}
Thus the set of Hammersley nets $\{\mathcal{H}_{m,b}\}_{m \in \NN}$ is not quasi-uniform nor well-separated.
\end{proof}

The above proof finds a small, shifted elementary interval which contains the two points $\bsx_n,\bsx_k$. That is, by setting $\bsc = \left( m-1, m-1 \right)$, $\bsa = \left( b^{c_1-1}, b^{c_2-1} \right)$ and $\bse = (1,1)$, then we see that $\bsx_n, \bsx_k \in J_{\bsc, \bsa, \bse}$. This means that the Hammersley net is not $(m-1)$-separated. Thus, Lemma~\ref{lem:nonwell_is_smallsep} implies that the separation radius of the net is relatively small as compared to the covering radius.

\subsection{Some \texorpdfstring{$(0,2)$}{(0,2)}-sequences in base 2}
In \cite{HS19}, it was shown that any digital $(0,2)$-sequence in base $2$ is given by generating matrices of the form
\[ C_1=L_1 U \quad \text{and}\quad C_2=L_2PU,\]
where $L_1$ and $L_2$ are non-singular lower-triangular infinite matrices, $U$ is a non-singular upper-triangular infinite matrix, and $P$ denotes the upper triangular Pascal matrix, all over $\FF_2$, i.e.,
\begin{align}\label{eq:pascal}
P_{i,j} = \binom{j-1}{i-1},
\end{align}
where this value is equal to $0$ if $j < i$.

When we only consider the first $2^m$ points, multiplying by $U$ from the right only affects the order of points and does not change the covering radius or separation radius of the point set. 
Therefore, we focus on the case $U=I$ with $I$ being the infinite identity matrix. 
Note that $L_1$ and $L_2$ correspond to linear scrambling. 

In the following theorem, we prove that a digital $(0,2)$-sequence in base $2$ is not well-separated when scrambling is applied individually by either $L_1$ or $L_2$ but not both. 
When $L_1=L_2=I$, the resulting sequence coincides with the two-dimensional Sobol' sequence.
This means that the result below generalizes the one shown in \cite{G24a}.
It remains open, however, whether there exists a pair 
$(L_1,L_2)$ that makes the sequence well-separated (or $\{Q_b^m\}_{m \in I}$ a well-separated family on certain index set $I$).

\begin{theorem}
Let $L$ be any non-singular lower-triangular infinite matrix over $\FF_2$. Let $w$ be any non-negative integer and $m=2^{w}$. If $(C_1,C_2)=(L,P)$ or $(C_1,C_2)=(I,LP)$, then 
    \[ q_\infty(Q_{2^m})\leq \frac{1}{2^{m}} .\]
This bound implies that the corresponding digital $(0,2)$-sequence in base $2$ is not well-separated.
\end{theorem}

\begin{proof}
The proof follows similar arguments to those in \cite{G24a}. Let $L = (\ell_{i,j})_{i,j \ge 1}$, where $\ell_{i,i} = 1$ and $\ell_{i,j} = 0$ for all $j > i$. Consider the case $(C_1, C_2) = (L, P)$. Then 
\begin{equation*}
\bsx_1 = \left(\frac{\ell_{1,1}}{2} + \frac{\ell_{2,1}}{2^2} + \cdots, \frac{1}{2} \right).
\end{equation*}

For any integer $i$ with $1\leq i\leq m=2^w$, whose dyadic expansion has digits given by 
$\vec{i} = (i_0,i_1,\dots,i_{w-1}, 0, 0, \ldots)^{\top}$ as in Definition~\ref{def:digitalseq},
Lucas's theorem implies
\[ P_{i,m}\equiv \binom{2^w-1}{i-1}\equiv \prod_{j=0}^{w-1}\binom{1}{i_j}\equiv 1 \pmod 2. \]
Similarly, excluding the case $i=m$, we have
\[ \sum_{j=1}^{m}P_{i,j}\equiv \binom{2^w}{i} \equiv \binom{1}{0}\times \prod_{j=0}^{w-1}\binom{0}{i_j}\equiv 0 \pmod 2.\]
This equivalence implies that, among $m$ entries $(P_{i,j})_{1\leq j\leq m}$, 
the number of entries with $P_{i,j}=1$ is even for any $i<m=2^w$.

Let $k=2^{m-1}+1$. Since $\vec{k} = (1,0,\ldots,0,1, 0, 0, \ldots)^\top$ we have
\begin{align*}
    \vec{x}_{k,1} & = \left(\ell_{1,1}, \ell_{2,1},\ldots, \ell_{m-1,1}, \ell_{m,1} \oplus \ell_{m,m}, \ell_{m+1,1} \oplus \ell_{m+1,m}, \ldots \right)^{\top} \quad \text{and} \\
    \vec{x}_{k,2} & = P \times\left(1,0,\ldots,0,1, 0, 0, \ldots \right)^{\top} \\
    & = (P_{1,1}+P_{1,m}, P_{2,1}+P_{2,m}, \ldots, P_{m,1}+P_{m,m}, \ldots)^{\top} \\
    & = (0,1,\ldots,1, 0, 0, \ldots )^\top.
\end{align*}
Thus 
\[ \bsx_k=\left(\frac{\ell_{1,1}}{2}+\cdots + \frac{\ell_{m-1,1} }{2^{m-1}} + \frac{\ell_{m,1} \oplus \ell_{m,m}}{2^m} + \frac{\ell_{m+1,1} \oplus \ell_{m+1,m}}{2^{m+1}} + \cdots, \frac{1}{2}-\frac{1}{2^m}\right),\]
which shows that
\[ q_\infty(Q_{2^m})\leq \frac{1}{2}\left\| \bsx_1-\bsx_k\right\|_{\infty} \le \frac{1}{2^{m}}.\]

Let us move on to the case $(C_1,C_2)=(I,LP)$. Since $P$ is a non-singular upper triangular matrix, we have
\begin{equation*}
\bsx_1 = \left( \frac{1}{2}, \frac{\ell_{1,1}}{2} + \frac{\ell_{2,1}}{2^2} + \cdots  \right).
\end{equation*}

Let $k=2^{m}- 2$. Since $\vec{k} = (0,1, 1, \ldots,1,1, 0, 0, \ldots )^\top $, we have
\begin{align*}
    \vec{x}_{k,1} & = \left(0,1, 1, \ldots,1,1, 0, 0, \ldots \right)^{\top}\quad \text{and} \\
    \vec{x}_{k,2} & = L \times P \times \left(0,1, 1, \ldots,1,1, 0, 0, \ldots  \right)^{\top} \\
    & = L \times \left(P_{1,1}+ \sum_{j=1}^{m}P_{1,j}, P_{2,1}+ \sum_{j=1}^{m}P_{2,j}, \ldots, P_{m,1}+ \sum_{j=1}^{m}P_{m,j}, \ldots \right)^{\top}\\
    & = L (1,0,\ldots,0,1, 0, 0, \ldots  )^\top \\ 
    & = (\ell_{1,1}, \ell_{2,1}, \ldots, \ell_{m-1,1}, \ell_{m,1} \oplus \ell_{m,m}, \ell_{m+1,1} \oplus \ell_{m+1,m}, \ldots )^\top .
\end{align*}
Thus 
\[ \bsx_k=\left( \frac{1}{2}-\frac{1}{2^m}, \frac{\ell_{1,1}}{2} + \cdots + \frac{\ell_{m-1,1}}{2^{m-1}} + \frac{\ell_{m,1} \oplus \ell_{m,m}}{2^m} + \frac{\ell_{m+1,1} \oplus \ell_{m+1,m}}{2^{m+1}} + \cdots  \right),\]
which shows that
\[ q_\infty(Q_{2^m})\leq \frac{1}{2}\left\| \bsx_1-\bsx_k\right\|_{\infty} \le \frac{1}{2^{m}}. \]

Thus, the corresponding digital $(0,2)$-sequence in base $2$, either with $(C_1,C_2)=(L,P)$ or $(C_1,C_2)=(I,LP)$, is not well-separated.
\end{proof}

\subsection{Faure sequence}\label{sec:Faure}

In this subsection, we show that the Faure sequence is not quasi-uniform.
Since the Faure sequence in base $2$ is the same as the Sobol' sequence in dimension $d=2$, this result also generalizes the one shown in \cite{G24a}.

Let $P$ be the upper triangular Pascal matrix as given in \eqref{eq:pascal}, now considered over $\FF_b$ for a prime $b$.
The $b$-dimensional Faure sequence in base $b$ is a digital $(0,b)$-sequence in base $b$ with the generating matrices
\[
I,P,P^2, \dots, P^{b-1}.
\]
It is shown in \cite{Faure1982dds} that the entries of the generating matrices are given by
\begin{equation}\label{eq:Faure_matrix_elements}
(P^k)_{i,j} \equiv k^{j-i}\binom{j-1}{i-1} \pmod{b}
\end{equation}
if $j \ge i$ and otherwise $0$.

On the separation radius, we have the following.
\begin{theorem}\label{thm:Faure-near-points}
Let $b$ be a prime and $(\bsx_i)_{i \ge 0} \subset [0,1)^b$ be the Faure $(0,b)$-sequence in base $b$.
Let $w$ be a positive integer, $m = (b-1)b^w$ and
$n := b^m - b$.
Then we have
\[
\|\bsx_1 - \bsx_n\|_\infty \le b^{-(b^w-1)}.
\]
In particular,
the separation radius of the set
$Q_{b^m} := \{\bsx_0, \dots, \bsx_{b^m-1}\}$ is bounded from above by
\[
q_\infty(Q_{b^m}) \le \frac{1}{2}\|\bsx_1 -\bsx_n\|_\infty
\le \frac{b}{2}(b^{m})^{-\tfrac{1}{b-1}}.
\]
\end{theorem}

Since the separation radius decays with order $(b^m)^{-\tfrac{1}{b-1}}$, which is faster than the optimal, desired rate $(b^m)^{-\tfrac{1}{b}}$, by combining Theorem~\ref{thm:Faure-near-points} and \eqref{eq:CoveringRadius-tmsnet}, we obtain the following main result in this section.

\begin{corollary}
The Faure $(0,b)$-sequence in base $b$ is not a quasi-uniform sequence.
\end{corollary}

Note that this result, for instance, does not imply that the Faure $(0,d)$-sequence for $2 < d < b$ in base $b$ is not quasi-uniform.

To prove Theorem~\ref{thm:Faure-near-points},
we first give some technical lemmas.

\begin{lemma}\label{lem:exp_sum}
Let $2 \le k \le b-1$ and $w \ge 1$ be integers.
Then we have
\[
\sum_{c=0}^{b-2}k^{cb^w}
\equiv 0 \pmod{b}.
\]
\end{lemma}

\begin{proof}
By Fermat's little theorem,
we have $k^{b} \equiv k \pmod{b}$ and thus $k^{b^w} \equiv k \pmod{b}$.
Hence
\[
\sum_{c=0}^{b-2}k^{cb^w}
\equiv
\sum_{c=0}^{b-2}k^{c}
= \dfrac{k^{b-1}-1}{k-1}
\equiv 0 \pmod{b},
\]
where, in the last equivalence, we use $k \not \equiv 1 \pmod{b}$ and Fermat's little theorem again.
\end{proof}

\begin{lemma}\label{lem:Lucas-modified}
Let $b$ be a prime and
$0 \le c < b$, $0 \le t < b^w$ and $0 \le u < b^w$ be integers.
Then we have
\[
\binom{cb^w + t}{u}
\equiv \binom{t}{u} \pmod{b}.
\]
\end{lemma}

\begin{proof}
Let us denote the $b$-adic expansions of $cb^w+t$ and $u$ by
\[ cb^w+t = v_0+v_1 b+\cdots + v_Lb^{L} \quad \text{with $v_0,v_1,\ldots,v_L\in \{0,\ldots,b-1\}$,}\]
and
\[ u = u_0+u_1 b+\cdots + u_Lb^{L} \quad \text{with $u_0,u_1,\ldots,u_L\in \{0,\ldots,b-1\}$,}\]
respectively, where we permit leading zeros in the expansions.
By the assumptions, $v_0+v_1 b+\cdots + v_{w-1}b^{w-1}$ is equal to the $b$-adic expansion of $t$
and $u_L = u_{L-1} \cdots = u_{w} = 0$ holds.
Thus, by using Lucas's theorem twice, we have
\[
\binom{cb^w + t}{u}
\equiv \prod_{i=0}^{L} \binom{v_i}{u_i}
\equiv \prod_{i=0}^{w-1} \binom{v_i}{u_i}
\equiv \binom{t}{u} \pmod{b}.
\]
Thus we are done.
\end{proof}

\begin{lemma}\label{lem:Faure_acc_sum}
Let $w \ge 0$ be an integer and $m=(b-1)b^w$.
Then the following holds true.
\begin{enumerate}
\item[\normalfont (i)]
For any integer $1 \le i \le b^w-1$, we have
\[
\sum_{j=1}^{m} P_{i,j} \equiv 0 \pmod{b}.
\]
\item[\normalfont (ii)]
For any integers $2 \le k \le b-1$ and $1 \le i \le b^w$,
we have
\[
\sum_{j=1}^{m} (P^k)_{i,j} \equiv 0 \pmod{b}.
\]

\end{enumerate}

\end{lemma}

\begin{proof}
By applying \eqref{eq:Faure_matrix_elements}, the hockey-stick identity
and Lemma~\ref{lem:Lucas-modified} with $(c,t,u) =(b-1,0,i)$ in order,
we have
\[
\sum_{j=1}^{m} P_{i,j}
= \sum_{j=i}^{m} \binom{j-1}{i-1}
= \binom{m}{i}
\equiv \binom{0}{i}
\equiv 0 \pmod{b}
\] 
for any $1 \le i \le b^w-1$.
Thus we have shown the first claim.

Let us now assume $2 \le k \le b-1$ and $1 \le i < b^w$.
For $1 \le j \le m$, let $j-1=c(j)b^w+t(j)$ with $0 \le c(j) \le b-2$ and $0 \le t(j) < b^w$.
By Lemma~\ref{lem:Lucas-modified} we have
\[
\binom{j-1}{i-1} \equiv \binom{t(j)}{i-1} \pmod{b}
\]
and, in particular, this equals $0$ modulo $b$ if $0 \le t(j) < i-1$.
Thus it follows from \eqref{eq:Faure_matrix_elements} that
\begin{align*}
\sum_{j=1}^{m} (P^k)_{i,j}
&\equiv \sum_{c=0}^{b-2}\sum_{t=i-1}^{b^{w} -1} k^{cb^w+t+1-i}\binom{t}{i-1}\\
&\equiv \sum_{c=0}^{b-2}k^{cb^w}\sum_{t=i-1}^{b^{w}-1} k^{t+1-i}\binom{t}{i-1}
\equiv 0
\pmod{b},
\end{align*}
where we used Lemma~\ref{lem:exp_sum} in the last equivalence.
Thus we have proved the second claim.
\end{proof}

We now give a proof of Theorem~\ref{thm:Faure-near-points}.

\begin{proof}[Proof of Theorem~\ref{thm:Faure-near-points}]
We compute each coordinate of $\bsx_n = (x_{n,1},\dots,x_{n,b})$.
For the first coordinate,
since $\vec{n} = (\underbrace{0,b-1,\dots,b-1}_{m}, 0, 0, \ldots)^\top$
we have
\begin{equation}\label{eq:Faurethm-1}
x_{n,1} = 0 + \dfrac{b-1}{b^2} + \cdots + \dfrac{b-1}{b^m} = \dfrac{1}{b} - \dfrac{1}{b^m}.    
\end{equation}

For the $k$-th coordinate with $2 \le k \le b-1$,
by Lemma~\ref{lem:Faure_acc_sum} we have
\[
P^{k-1} (\overbrace{ b-1,\dots,b-1}^{m}, 0, 0, \ldots)^\top 
= \begin{cases}
(\overbrace{0,\dots,0}^{b^w-1},*,*,\dots)^\top & \text{if }k=2,\\
(\overbrace{0,\dots,0}^{b^w},*,*,\dots)^\top & \text{if } 3 \le k \le b-1,
\end{cases}
\]
where $*$ represents some integer between $0$ and $b-1$.
Further we have
\[
P^{k-1}  (b-1,0,0,\dots )^\top 
= ((b-1)P^{k-1}_{1,1},0,0,\dots )^\top
= (b-1,0,0,\dots )^\top.
\]
Using these equalities, for $k=2$ we have
\begin{align*}
P^{k-1} \vec{n}
= (\overbrace{0,\dots,0}^{b^w-1},*,*,\dots)^\top \ominus (b -1,0,0\dots)^\top
= (1,\overbrace{0,\dots,0}^{b^w-2},*,*,\dots)^\top
\end{align*}
and thus we have
\begin{equation}\label{eq:Faurethm-2}
\dfrac{1}{b} \le x_{n,2} \le \dfrac{1}{b} + \sum_{i=b^w}^{m} \dfrac{b-1}{b^i}
< \dfrac{1}{b} + \dfrac{1}{b^{b^w-1}}.
\end{equation}
Similarly, for $3 \le k \le b-1$ we have
\begin{equation}\label{eq:Faurethm-k}
\dfrac{1}{b} \le x_{n,k} < \dfrac{1}{b} + \dfrac{1}{b^{b^w}}.    
\end{equation}
Since $\bsx_1 = (1/b, \dots, 1/b)$,
by combining \eqref{eq:Faurethm-1}, \eqref{eq:Faurethm-2} and \eqref{eq:Faurethm-k}
we obtain the desired result.
\end{proof}

\subsection{Some Fibonacci polynomial lattice}\label{sec:Fib-poly-lat}

Polynomial lattice point sets are digital nets whose construction is based on rational functions over finite fields,
which are analogous to lattice point sets \cite{N92}.
Since the Fibonacci lattices are known to be low-discrepancy and will also be proven to be quasi-uniform in \cite{DGLPSxx},
it is natural to ask whether the Fibonacci polynomial lattices, their polynomial counterparts, are also quasi-uniform (they are known to be low-discrepancy).
In this section, we show this is not always true, 
since the mesh ratios of some of them
in base~$2$ are not uniformly bounded.

The polynomial lattice point set in base $2$ for $d=2$ is defined as follows.
\begin{definition}\label{def:polylat}
Let $m\in \NN$.
For an integer $0 \le n < 2^m$ with binary expansion $n = n_0 + 2n_1 + \cdots + 2^{m-1}n_{m-1}$,
let $n(x) = n_0 + n_1x + \cdots + n_{m-1}x^{m-1} \in \FF_2[x]$.
For polynomials $p(x),q_1(x),q_2(x) \in \FF_2[x]$ with $m := \deg(p)$,
the polynomial lattice point set in base $2$, denoted by $Q(p,(q_1,q_2))$, is given by
$\{\bsx_n \mid 0 \le n < 2^m\}$ with
\[
\bsx_n := \left( \nu_m \left( \frac{n(x)q_1(x)}{p(x)}\right), \nu_m \left( \frac{n(x)q_2(x)}{p(x)}\right) \right),
\]
where we write
\[
\nu_m\left( \sum_{i=w}^{\infty}t_i x^{-i}\right) = \sum_{i=\max(w,1)}^{m}t_i b^{-i}\in [0,1).
\]
\end{definition}
We use the following properties \cite[Section~4.4]{N92}.
The polynomial lattice point set $Q(p,(q_1,q_2))$ is a digital net. The entries of the generating matrices $F_1,F_2$ are given by
\begin{equation}\label{eq:genmat_polylat}
(F_j)_{k,\ell} = u_{k+\ell -1}^{(j)} \qquad (1 \le k,\ell \le m,\, j=1,2),
\end{equation}
where the $u_i^{(j)}$'s are defined by the expansion in the set of formal Laurent series $\FF_{2}((x^{-1}))$ as
\[
\dfrac{q_j(x)}{p(x)} = \sum_{i=w_j}^{\infty}u_i^{(j)} x^{-i}.
\]
Furthermore, let $p(x),q(x) \in \FF_2[x]$ with $\gcd(p,q) = 1$.
Assume that the continued fraction expansion is given by
\[
\dfrac{q(x)}{p(x)} = [a_0(x);a_1(x),\dots,a_{\ell}(x)]
\]
and let $A := \max(\deg(a_1), \dots, \deg(a_\ell))$.
Then the $t$-value of the polynomial lattice $Q(p, (1,q))$
is equal to $A - 1$ (see \cite[Theorem~4.46]{N92}).

To define the Fibonacci polynomial lattices, we define the Fibonacci polynomials over $\FF_2$
recursively by
\begin{equation}\label{eq:fibpoly}
f_1(x) = 1,
\qquad f_2(x) = x,
\qquad f_{n+2}(x) = xf_{n+1}(x) + f_n(x).    
\end{equation}
Note that $\deg(f_n) = n-1$.
It is well-known (see, e.g., \cite[(3.3)]{B70}) that
\begin{equation}\label{eq:Fib-iter}
f_{m+n}(x) = f_{m+1}(x)f_{n}(x) + f_{m}(x)f_{n-1}(x)
\end{equation}
holds for any integers $m,n$.
By using Eq.~\eqref{eq:Fib-iter} with $m+1=n=2^k$ and $m=n=2^k$,
we can show the following lemma by induction.
\begin{lemma}
Let $k \ge 1$ be an integer.
Then we have
\begin{align}\label{eq:fib_poly}
f_{2^{k}-1}(x) = \sum_{\ell=1}^k x^{2^k - 2^\ell} \quad \text{and}\quad f_{2^{k}}(x) & = x^{2^k-1}. 
\end{align}
\end{lemma}

Then the Fibonacci lattice point set in base $2$ with $2^m$ points is defined to be the polynomial lattice 
\[
Q(f_{m+1}(x), (1,f_m(x))).
\]
Since we have the continued fraction expansion $f_m(x)/f_{m+1}(x) = [0;x,x,\ldots,x]$ for any $m \in \NN$,
the Fibonacci polynomial lattice point set with $2^m$ points is a digital $(0,m,2)$-net in base $2$.
On the other hand, the corresponding family of point sets turn out to be not well-separated as follows.

\begin{theorem}\label{thm:Fiblat_sep}
Let $k \ge 4$ be an integer and define $m := 2^k-1$,
$m' = 2^{k-3}-1$.
Let $Q_m = Q(f_{m+1},(1,f_m)) = \{\bsx_0, \dots,\bsx_{2^{m}-1}\}$ be the Fibonacci polynomial lattice with $2^m$ points. 
Then,
for $n = 2^{m-m'-2} - 2^{m'-1}$ and
$n' = 2^{m-m'-2} + 2^{m'-1}$
we have
\[
\|\bsx_{n}-\bsx_{n'}\|_{\infty}
= 2^{-m+m'}.
\]
In particular, by letting $N=2^m$ be the number of points in $Q_m$, the separation radius of $Q_m$ is of order
\[
q_\infty(Q_m) \le 2^{-15/8} N^{-7/8}. 
\]
\end{theorem}

This theorem directly implies the following.
\begin{corollary}
The set of the Fibonacci polynomial lattices $\{Q(f_{m+1},(1,f_m))\}_{m \in \NN}$ is not a quasi-uniform family on $\NN$.
\end{corollary}
Note that this result does not exclude the possibility that there is a subset $I \subset \mathbb{N}$ such that $\{Q(f_{m+1}, (1, f_m)) \}_{m \in I}$ is a quasi-uniform family.

\begin{proof}[Proof of Theorem~\ref{thm:Fiblat_sep}]
In this proof, we denote by $v[i]$ the $i$-th entry of a vector $v$
and by $C[i,j]$ the $(i,j)$-th entry of a matrix $C$ for visibility.

Since the denominator of the polynomial lattice is $f_{m+1}(x) = x^{m}$ by \eqref{eq:fib_poly},
if follows from \eqref{eq:genmat_polylat} that the generating matrix of the first coordinate $F_1$ is the anti-diagonal identity matrix, i.e.,
\[
F_1 = \begin{pmatrix}
0 & \ldots & \ldots & 0 & 1 \\
0 & \ldots & 0 & 1 & 0 \\
\vdots & \iddots & \iddots & \iddots & \vdots \\
0 & 1 & 0 & \ldots & 0 \\
1 & 0 & \ldots & 0 & 0
\end{pmatrix}.
\]
Thus $x_{n,1} = n2^{-m}$ holds for any integer $n\in \{0,1,\ldots,2^m-1\}$
and we have
\[
|x_{n,1}-x_{n',1}| = 2^{-m}(n'-n) = 2^{-m+m'}.
\]

Let us move on to the second coordinate.
Since it follows from \eqref{eq:fib_poly} that
\[ \frac{f_{m}(x)}{f_{m+1}(x)}=\sum_{\ell=1}^{k}x^{-2^{\ell}+1},\]
the generating matrix of the second coordinate $F_2$ is given by
\begin{equation}\label{eq:coeff-F2}
F_2[i,j] =
\begin{cases}
1 & \text{$i+j=2,4,8,16,\cdots$ (i.e., a power of two),}\\
0 & \text{otherwise.}
\end{cases}
\end{equation}

We first compute $x_{n,2}$ via $F_2 \vec{n}$.
Considering the dyadic expansion of $n = 2^{m-m'-2} - 2^{m'-1}$,
we have $\vec{n}[j] = 1$ if and only if
$m' \le j < m-m'-1$ and $\vec{n}[j] = 0$ otherwise.
Thus it follows from \eqref{eq:coeff-F2} that,
for each $1 \le i \le m$,
$(F_2 \vec{n})[i]$ equals the parity of the number of $j$ which satisfy $m' \le j = 2^t-i < m-m'-1$ for some positive integer $t$.
For fixed $t$, under the condition $1 \le i \le m$, we have
\begin{align*}
m' \le 2^t-i < m-m'-1
\iff
\begin{cases}
i = 1 & t=k-3, \\
1 \le i \le 2^{k-3} + 1  & t=k-2, \\
1 \le i \le 2^{k-1}- 2^{k-3} + 1 & t=k-1, \\
2^{k-3} + 1 < i \le 2^k- 2^{k-3} + 1 & t=k,\\
\text{no such $i$ with $1 \le i \le m$} & \text{otherwise.}\\
\end{cases}
\end{align*}
Hence we have
\[
(F_2\vec{n})[i] = 1
\iff
i=1
\text{ or } 2^{k-1} - m'+ 1  \le i \le 2^k - m'  
\]
and thus
\[
x_{n,2}
= 2^{-1}
+ \sum_{i = 2^{k-1}-m'+1}^{2^k-m'}2^{-i}
= 2^{-1} + 
2^{-2^{k-1} + m' } - 2^{-2^{k} +m'}.
\]

We now compute $x_{n',2}$ via $F_2 \vec{n}'$.
Since we have $\vec{n}'[i] = 1$ if and only if $i = m'$ or $i=m-m'-1$,
$F_2 \vec{n}'$ is the sum of the $m'$-th column and the $(m-m'-1)$-th column of $F_2$.
From \eqref{eq:coeff-F2}, for the $m'$-th column, we have
\[
F_2[i,m'] = 1
\iff \text{$i+m'$ is a power of $2$}
\iff i = 2^{k-t}-m' \text{ for } t=0,1,2,3.
\]
Similarly, for the $(m-m'-1)$-th column we have
\[
F_2[i,m-m'-1] = 1
\iff \text{$i+m-m'-1$ is a power of $2$}
\iff i = m'+2.
\]
Thus, noting that $m'+2 = 2^{k-2}-m'$, we have
\[
(F_2 \vec{n}')[i] = 1
\iff i = 1,2^{k-1}-m',2^{k}-m'
\]
and thus
\[
x_{n',2}
= 2^{-1} + 2^{-2^{k-1}+m'} + 2^{-2^k+m'}.
\]
Combining these results,
we have 
\[
|x_{n,2}-x_{n',2}| = 2 \cdot 2^{-2^k+m'} = 2^{-m+m'}.
\]

Finally, regarding the separation radius of $Q_m$, it holds that
\[ q_{\infty}(Q_m)\leq \frac{1}{2}\|\bsx_{n}-\bsx_{n'}\|_{\infty} =  2^{-m+m'-1}=2^{-7m/8-15/8}.\]
Thus we are done.
\end{proof}

\begin{remark}
We have shown in 
Sections~\ref{sec:Faure} and \ref{sec:Fib-poly-lat} that the Faure sequence and certain sequences of Fibonacci polynomial lattice points are not quasi-uniform.
However, the values of $N$ for which the separation radius can be proved to be too small are very sparse, namely  
$N = b^{(b-1)b^w}$ in the former case and $N = 2^{2^k-1}$ in the latter.  
Thus we cannot rule out the possibility of quasi-uniformity along some subsequence.
A more refined analysis remains a topic for future research.

For the two-dimensional Sobol' sequence, the exact values of the separation radius for all $N = 2^m$ with $m \in \NN$ were recently obtained, showing that the mesh ratio is suboptimal and grows on the order of $O(N^{1/4})$ for all $N$~\cite{suzuki2025a}.  
\end{remark}

\begin{remark}
Finally, we briefly comment on possible routes toward quasi-uniform digital constructions. It is natural to ask whether standard randomization of digital nets, such as digital shifts or Owen-type scrambling, could restore quasi-uniformity.
While these operations preserve (and often improve) distributional uniformity, they do not by themselves enforce geometric separation.
Furthermore, it is well known that the separation radius of IID points is far from optimal and thus randomization may not help to increase the quasi-uniformity.
We leave a detailed investigation of these directions to future work.
\end{remark}

\bibliographystyle{plain}
\bibliography{ref.bib}

\end{document}